\documentclass{imsart}

\usepackage{latexsym}
\usepackage{amsthm,amsmath}
\usepackage{enumerate,amssymb,dsfont,pifont, color}

\arxiv{math.PR/1709.06756}

\startlocaldefs

\theoremstyle{plain}
\newtheorem{theorem}{Theorem}[section]
\newtheorem{corollary}[theorem]{Corollary}
\newtheorem{lemma}[theorem]{Lemma}
\newtheorem{proposition}[theorem]{Proposition}

\theoremstyle{definition}
\newtheorem{definition}[theorem]{Definition}

\theoremstyle{remark}
\newtheorem{remark}[theorem]{Remark}
\newtheorem{remarks}[theorem]{Remarks}
\newtheorem{example}[theorem]{Example}

\renewcommand{\leq}{\leqslant}
\renewcommand{\geq}{\geqslant}

\newcommand{\bbc}{\mathbb{C}}
\newcommand{\bbr}{\mathbb{R}}

\newcommand{\bbp}{\mathbb{P}}
\newcommand{\bbe}{\mathbb{E}}

\newcommand{\bbn}{\mathbb{N}}
\newcommand{\bbk}{\mathbb{K}}

\newcommand{\cc}{\mathcal{C}}
\newcommand{\cf}{\mathcal{F}}

\newcommand{\ca}{\mathcal{A}}

\newcommand{\pf}{\longrightarrow}

\newcommand{\abs}[1]{\left| #1 \right|}
\newcommand{\norm}[1]{\left\| #1 \right\|}

\newcommand{\loi}[2]{\left] #1 , #2 \right]}                             
\newcommand{\roi}[2]{\left[ #1 , #2 \right[}
\newcommand{\oi}[2]{\left] #1, #2 \right[}
\newcommand{\ci}[2]{\left[ #1, #2 \right]}

\newcommand{\rosi}[2]{\roi{[#1}{#2[}}

\newcommand{\csi}[2]{\ci{[#1}{#2]}}

\newcommand{\E}{\widetilde{E}}
\newcommand{\rd}{\widetilde{\bbr^d}}
\newcommand{\change}[1]{\textcolor{black}{#1}}

\endlocaldefs

\begin{document}

\begin{frontmatter}

\title{The Fourth Characteristic of a Semimartingale}
\runtitle{Fourth Characteristic}


\author{\fnms{Alexander} \snm{Schnurr}\corref{}\ead[label=e1]{schnurr@mathematik.uni-siegen.de}}
\address{Alexander Schnurr \\ University Siegen,
Department of Mathematics \\ Walter-Flex-Street 3, 57068 Siegen
\\ \printead{e1}}

\runauthor{Alexander Schnurr}

\begin{abstract}
We extend the class of semimartingales in a natural way.
This allows us to incorporate processes having paths that leave the state space $\bbr^d$.
In particular Markov processes related to sub-Markovian kernels\change{, but also non-Markovian processes with path-dependent behavior.}
By carefully distinguishing between two killing states, we are able to introduce a fourth semimartingale characteristic which generalizes the fourth part of the L\'evy quadruple. 
Using the probabilistic symbol, we analyze the close relationship between the generators of certain Markov processes with killing and their (now four) semimartingale characteristics. 
\end{abstract}


\end{frontmatter}

\section{Introduction} 

Let $Z=(Z_t)_{t\geq 0}$ be a L\'evy process.     
Assume in addition that $Z$ is conservative, that is, $\bbp^z(Z_t\in\bbr^d)=1$ for every $t\geq 0$ and each starting point $z\in \bbr^d$. 
It is a well known fact that the characteristic function $\varphi_{Z_t}:\bbr^d\to\bbc$
can be written as $(t\geq 0)$
\begin{align} \label{charexp}
  \varphi_{Z_t}(\xi)=\bbe^z\left( e^{i(Z_t-z)'\xi} \right) = \bbe^0\left( e^{iZ_t'\xi} \right) =  e^{-t\psi(\xi) }
\end{align}
where 
\begin{align} \begin{split}\label{lkf}
  \psi(\xi)&= -i \ell'  \xi + \frac{1}{2} \xi'Q \xi
      - \int_{\bbr^d\backslash\{0\}} \left(e^{i y'\xi} -1 - i y'\xi \cdot \chi(y) \right) \,  N(dy).
\end{split}\end{align}
Here, $\ell\in\bbr^d$, $Q$ is a positive semidefinite matrix, $N$ the so called L\'evy measure (cf. \cite{sato} (8.2)) and $\chi:\bbr^d\to \bbr$ is a cut-off function. 
Every conservative L\'evy process is a semimartingale with respect to its natural filtration. 
The semimartingale characteristics $(B_t, C_t, \nu(\cdot,dt))$ are a generalization of the L\'evy triplet $(\ell, Q, N)$.


The characteristic exponent $\psi$ of each L\'evy process is a continuous negative definite function in the sense of Schoenberg (cf. \cite{bergforst}, Section 7). 
Each function of this class can be represented in the following way:
\begin{align} \begin{split}\label{lkfa}
  \phi(\xi)&= a -i \ell'  \xi + \frac{1}{2} \xi'Q \xi
      - \int_{\bbr^d\backslash\{0\}} \left(e^{i y'\xi} -1 - i y'\xi \cdot \chi(y) \right) \,  N(dy).
\end{split}\end{align}
Even with the additional component $a>0$ one can associate a stochastic process $\widetilde{Z}$ with this characteristic exponent via \eqref{charexp}. This process is the L\'evy process $Z$ associated with $(\ell,Q,N)$ with the following modification: with $a$ we associate a killing time, which is exponentially distributed with parameter $a$ and independent of $Z$. The new process $\widetilde{Z}$ (with killing) behaves like $Z$, but as soon as the killing time is reached, it jumps to $\Delta$.
It will become handy to write this as
\[
  \widetilde{Z_t}=Z_t+K_t
\]
where $K=(K_t)_{t\geq 0}$ denotes the `killing process' which only attends the values $\{0,\Delta\}$. 
Let us emphasize already here that $\Delta$ denotes \emph{one} distinguished killing state and we set $\Delta+r:= \Delta$ for every $r\in\bbr^d$. In the case of L\'evy processes, $Z$ and $K$ are independent and one killing state is sufficient. Obviously, there exist more interesting processes admitting (state-space-)dependent killing. Let us mention that this very simple class of L\'evy processes with killing is not included in the classical semimartingale setting. In \cite{symbolkillingopenset} we have introduced a class of semimartingales admitting a \emph{predictable} killing, but even in this framework L\'evy processes with killing are not included. This is somehow not satisfactory since they perfectly fit into the framework of sub-Markovian kernels and hence Markov processes which are in turn closely linked to semimartingales (cf. Cinlar et al. \cite{vierleute}). 

We have seen that it is canonical to associate the quadruple $(a,\ell, Q,N)$ with a L\'evy process with killing. This allows us to take the whole class of continuous negative definite functions into account (cf. \cite{bergforst} Theorem 10.8.). 
In the semimartingale framework there exists by now no equivalent concept to the fourth part $a$ of the L\'evy quadruple. It is the aim of this paper to provide a natural extension of the class of semimartingales along with a fourth characteristic. 

Since the three characteristics have become canonical, let us give a second motivation, why it is useful to introduce a fourth characteristic. 
There is an intimate relationship between Markov processes and semimartingales which has been studied in \cite{vierleute}, \cite{symbolkillingopenset} and \cite{che-fil-yor}. Let us elaborate on this relationship in the case of Feller processes with sufficiently rich domain. 

Let $\ca$ be the generator of the Feller semigroup. It is a well known fact (\cite{niels1} Section 4.5 and \cite{courrege}) that if the test functions $C_c^\infty(\bbr^d)$ are contained in the domain $D(\ca)\subseteq C_0(\bbr^d)$ of $\ca$,
this operator can be written as
\begin{align} \label{pseudo}
    \ca u(x)= - \int_{\bbr^d} e^{ix'\xi} q(x,\xi) \hat{u}(\xi) \, d\xi \hspace{1cm}  (u\in C_c^\infty(\bbr^d))
\end{align}
where $\hat{u}(\xi)=1/(2\pi)^d\int e^{-iy'\xi}u(y) dy$ denotes the Fourier transform and $q:\bbr^d \times \bbr^d \pf \bbc$ is locally
bounded and for fixed $x$ a continuous negative definite function in the
co-variable, that is, 
\begin{align} \begin{split}\label{lkfax}
  q(x, \xi)&= a(x) -i \ell(x)'  \xi + \frac{1}{2} \xi'Q(x) \xi
      - \int_{\bbr^d\backslash\{0\}} \left(e^{i y'\xi} -1 - i y'\xi \cdot \chi(y) \right) \,  N(x,dy).
\end{split}\end{align}
This function $q$ is called (functional analytic) symbol of the Feller process. Usually it is directly assumed that $a=0$ (cf. \cite{symbolkillingopenset}, \cite{schilling98}, \cite{BoetSchi2009}, \cite{levymatters3}). 
Hence, non-predictable killing is excluded.
In \cite{symbolkillingopenset} we have shown that under this additional assumption every such Feller process is a semimartingale and that there is a close relationship between the generator and the semimartingale characteristic (cf. in this context \c{C}inlar et al. \cite{vierleute}). Subsequently, we establish a framework that allows to handle the extended setting. In this framework Feller processes with sufficiently rich domain are semimartingales, \change{even if they jump to the point-of-no-return $\partial$ in a non-predictable way}. Moreover, these Feller processes belong to the natural extension of what is often called homogeneous diffusion with jumps (cf. \cite{jacodshir} Section III.2c). 

The reader might wonder why the fourth characteristic has been overlooked for quite a long time. Let us try to give a partial answer: 
\change{Starting with the L\'evy example from above, the fourth characteristic should describe a kind of `local killing rate'. Let $X$ be the solution of some martingale problem or be given by a family of sub-Markovian kernels. Writing down the state-space dependent killing rate at zero for such a process in a straightforward way one gets
\begin{align} \label{killingratepartial}
  \lambda^x:=\lim_{h\downarrow 0} \frac{\bbp^x(X_h = \partial)}{h}.
\end{align}
The point $\partial$ can be reached in different ways: jump to $\partial$ as soon as a certain value in space or time is reached, instant killing after an exponential waiting time, an accumulation of jumps having higher-and-higher jump intesity after each jump etc. While the last phenomenon can be described by the classical three characteristics (here, the third one). This is not the case for the other phenomena.
Therefore, one has to separate the killing time. Doing this in the canonical way (predictable vs. totally inaccessible) leads nowhere. In fact one has to separate between explosion vs. everything else. The explosion part can be described by the classical characteristics, while the remainder part is described by the fourth characteristic. Together this yields a full description of the paths leaving $\bbr^d$, allowing e.g. for a general representation result (cf. Theorem 2.5.). }

The notation closely follows \cite{jacodshir}. By $E$ we denote a closed subset of $\bbr^d$ and $\partial$ is the point-of-no-return which will be separated into two points subsequently. 
A function $\chi:\bbr^d\to\bbr$ is called cut-off
function if it is Borel measurable, with compact support and equal
to one in a neighborhood of zero. In this case $h(y):= \chi(y)\cdot
y$ is a truncation function in the sense of \cite{jacodshir}.
We will work on the canonical space, hence $\Omega$ denotes the space of c\`adl\`ag functions $\omega:\bbr_+\to E\cup \{\partial\}$ such that $\omega(t-)=\partial$ or $\omega(t)=\partial$ implie $\omega(u)=\partial$ for $u\geq t$. Analogously for $\Delta$ and $\infty$. As usual $X_t(\omega):=\omega(t)$ for $t\geq 0$, 
\begin{align*}
  \cf^X:=\sigma(X_s:s\geq 0) \text{ and } \cf^X_t:=\sigma(X_s:0\leq s \leq t). 
\end{align*}
Vectors $v$ are thought of as column vectors and we denote the transposed vector by $v'$. 

The paper is organized as follows: in the subsequent section we present the definitions and our main results (Theorem \ref{thm:markovsemimart} and Theorem \ref{thm:stoppedsymbol}). Some more examples from various areas of the theory of stochastic processes are analyzed shortly in Section 3. The final section consists of the proofs of the main results along with some side remarks. 

\section{Definitions and Main Results}


One of the main ideas in order to handle the fourth component is to distinguish carefully between two ways of killing, that is, of leaving the set $E\subseteq \bbr^d$ and to associate two different killing states ($\infty$ and $\Delta$) with the two types of killing. 
Consequently our processes live on $\E:=E\cup\{\infty,\Delta\}$. As a topological space $E_\infty:=E\cup \{\infty\}$  is the Alexandrov compactification of $E$. By adding another point, that is, by applying another Alexandrov extension, we get the isolated point $\Delta$. If not mentioned otherwise, every function $f$ on $E$ is extended to $\E$ by setting $f(\Delta)=f(\infty)=0$. 

Most of the time we will work with one probability measure $\bbp$, but sometimes we take the starting point into account, that is, we consider a stochastic basis $(\Omega, \cf ,(\cf_t)_{t\geq 0},\bbp^x)_{x\in \E}$. In this case it is always assumed that the process under consideration is \emph{normal}, i.e., $\bbp^x(X_0=x)=1$. 

\begin{definition}
The sequence of stopping times $(\sigma_n')_{n\in\bbn}$ given by ($n\geq 1$)
\[
  \sigma_n':=\inf \{ t \geq 0: \norm{X_t-x} \geq n \text{ or } \norm{X_{t-}-x} \geq n \},
\]
is called \emph{separating sequence}. Dealing with $(\bbp^x)_{x\in \E}$, all the stopping times depend on the starting point $x$. Since $x$ is most of the time fixed in our calculations we refrain from using a sub- or superscript $x$.
\end{definition} 
The sequence $(\sigma_n')_{n\in\bbn}$ will be used subsequently in order to separate explosion and other ways of killing:
		if a process is defined via the martingale problem or by a sub-Markovian family of kernels, usually a single ideal point, say $\partial$, is added to the state space. A posteriori it is possible to divide this ideal point into $\infty$ and $\Delta$. Let $\zeta^\partial$ be the stopping time, when $X$ leaves $E$. If $\sigma_n'(\omega)$ converges to $\zeta^\partial(\omega)$ without reaching it we set $X(\omega)=\infty$ on $\rosi{\zeta^\partial}{+\infty}$ otherwise we set it equal to $\Delta$. 
		In an analogous way we separate the stopping time $\zeta^\partial$. The ideal point is either reached by an explosion $\zeta^\infty$ or by a sudden killing (`jump') $\zeta^\Delta$:
		\begin{align} \begin{split} \label{killingtimes}
		  \zeta^\Delta&:=\begin{cases} \zeta^\partial & \text{,if } \sigma_n' = \zeta^\partial \text{ for some } n\in\bbn \\
			                     +\infty         & \text{,if } \sigma_n' < \zeta^\partial \text{ for all } n\in\bbn
						 \end{cases}  \\
		  \zeta^\infty&:= \begin{cases} \zeta^\partial & \text{,if } \sigma_n' < \zeta^\partial \text{ for all } n\in\bbn \\
			                     +\infty         & \text{,if } \sigma_n' = \zeta^\partial \text{ for some } n\in\bbn
						 \end{cases}  \\
		  \sigma_n&:= \begin{cases} \sigma_n' & \text{,if } \sigma_n' < \zeta^\partial \\
			                     +\infty         & \text{,if } \sigma_n' = \zeta^\partial
						 \end{cases}		
		\end{split} \end{align}
we obtain that $\{\zeta^\infty<+\infty\}$ and $\{\zeta^\Delta<+\infty\}$ are disjoint and that $\zeta^\infty$ is a predictable time with announcing sequence $\sigma_n\wedge n$. We set $\norm{\Delta}:=\norm{\infty}:=\norm{\partial}:=\infty$ (cf. Cherdito et al. \cite{che-fil-yor}).

\begin{definition} \label{def:killedprocess}
Let $E$ be a closed subset of $\bbr^d$. Let $X$ be a stochastic process on the stochastic basis 
$(\Omega, \cf^X ,(\cf^X_t)_{t\geq 0},\bbp)$ with values in $\E$. Let $\zeta^\infty$ be an explosion, that is, 
$\sigma_n\wedge n < \zeta^\infty$ for every $n\in\bbn$
and let $\zeta^\Delta$ be a stopping time. 

$X$ is called a \emph{process with killing} if
    \begin{align} \label{flowofprocess}
      X\cdot 1_{\rosi{0}{\zeta^\infty}}\subseteq \change{E} \text{, } X\cdot 1_{\rosi{\zeta^\infty}{\zeta^\Delta}}=\infty \text{ and } 
			X\cdot 1_{\rosi{\zeta^\Delta}{+\infty}}=\Delta.
    \end{align}
		
Here, we set as usual $\roi{\zeta^\infty(\omega)}{\zeta^\Delta(\omega)}=\emptyset$ if $\zeta^\infty(\omega) \geq \zeta^\Delta(\omega)$. In the special case that $\zeta^\Delta=+\infty$, we call $X$ a \emph{process with explosion}. 
\end{definition}

As we have pointed out above and as it will become clear in the subsequent sections, it is important to distinguish between $\infty$ and $\Delta$. To this end we define the following:
$\Delta+r=\Delta$ and $\infty+r=\infty$ for every $r\in\bbr^d$,
$\Delta\cdot s = \Delta$ and $\infty\cdot s = \infty$ for every $s\in\bbr$,
$\infty+\infty=\infty$,
$\Delta+\Delta=\Delta$
and finally
$\Delta+\infty=\Delta$.
The last point is in line with \eqref{flowofprocess} in the sense that we allow a transition from $\infty$ to $\Delta$, but not in the other direction. If we start with $\partial$ as described above, a transition from $\infty$ to $\Delta$ does not happen; in Example \ref{ex:firstexamples}(3) we encounter a situation where such a transition is natural. 


\change{We extend the class of semimartingales (cf. \cite{jacodshir} Definition I.4.21) in two steps: We call a process $X$ with explosion a \emph{semimartingale with explosion} if for the announcing sequence $\sigma_n\wedge n$ every $X^{(\sigma_n \wedge n)-}$ is a (classical) semimartingale.  
This definition could indeed be used for every predictable killing time. This is not needed subsequently. A reader who is interested in the details might consult the Appendix of \cite{symbolkillingopenset}. }

The most simple case of a semimartingale with predictable killing which is \emph{not} an explosion is the locally constant process given by 
$X_t^x=x$ on $\roi{0}{1}$ and being killed at time 1. Although this is a predictable time the process somehow jumps `all of a sudden' to infinity. In the proofs it turned out that it is more convenient and more natural to consider this kind of killing -- which is predictable but not an explosion -- in the $\Delta$ and $\zeta^\Delta$ context. Now we have to incorporate this (possibly) non-predictable killing into the semimartingale framework. 

\begin{definition} \label{def:gensemi}
A process with killing $\widetilde{X}=(\widetilde{X_t})_{t\geq 0}$ is called \emph{generalized semimartingale}, if it can be written in the following form:
$\widetilde{X_t}=X_t+K_t$
where $X=(X_t)_{t\geq 0}$ is a semimartingale with explosion and $K=(K_t)_{t\geq 0}$ is a killing process, that is, for a stopping time $\zeta^\Delta$
\[
  K_t=\Delta \cdot 1_{\rosi{\zeta^\Delta}{+\infty}}.
\]
\end{definition}


As in the classical case, the class of generalized semimartingales might often be too general to be used in application. \change{After some general results} we will describe some useful subclasses, like so called `autonomous processes' which are natural extensions of homogeneous diffusions with jumps.
Furthermore, we will show that a wide class of Markov processes is contained in this class.


Let $\widetilde{X}=X+K$ be a generalized semimartingale. The first three characteristics $(B,C,\nu)$ are defined (pre-)locally up to the predictable time $\zeta^\infty$: $X^{(\sigma_n\wedge n)-}$ is a classical semimartingale for every $n\in\bbn$. Hence, the three characteristics can be defined for these pre-stopped processes (cf. \cite{protter} Theorem II.6) in the classical way, cf. e.g. \cite{jacodshir} Section II.2. The characteristics of $X$ (and hence $\widetilde{X}$) are defined to be equal to these localized characteristics on $\rosi{0}{(\sigma_n\wedge n)}$ for every $n\in\bbn$. 

Following our motivating examples from above, the fourth characteristic should describe the local killing rate. In analogy to reliability theory and the theory of point processes we define the following: 

\begin{definition}
Let $\widetilde{X}$ be a generalized semimartingale with values in $\E$. The fourth characteristic $(A_t)_{t\geq 0}$ is the predictable compensator (cf. \cite{jacodshir} Theorem I.3.17) of the process $1_{\{ \widetilde{X}=\Delta \}}$ on $\csi{0}{\zeta^\Delta}$ that is the unique predictable process $A$ of finite variation such that $M$ defined via
\[
  M_t:= \left(1_{\{\widetilde{X}_t=\Delta\} }-A_t\right)^{\zeta^\Delta}
\]
is a local martingale. 
\end{definition}

Obviously the uniqueness of the process $A$ only holds up to the killing time $\zeta^\Delta$. Hence,  it is natural to define this last characteristic on $\csi{0}{\zeta^\Delta}$. 

\change{The reader might wonder, why $\Delta$ is not just included in the state space and the compensation of $\Delta$ in $\nu$. It is not the actual jump, which is important in compensation. It is the \emph{jump size}. If $\Delta$ is put inside the Euclidean space, one generates a non-canonical finite jump-size. If $\Delta$ is infinitely far away, the jump size is $+\infty$. Such a jump cannot be compensated. What should one subtract from a process with this jump in order to get a local martingale (stopping before this jump is also not an option, since the jump time might not be predictable). Hence, we follow the idea of the classical three characteristics: do not compensate the whole process, but something related to the process. For the second characteristic this is the continuous part of the square bracket; for the fourth characteristic it is $1_{\{\widetilde{X}_t\in\Delta\} }$. }

By the Definition \ref{def:gensemi} and Theorem II.2.34 of \cite{jacodshir} it is possible to directly derive a general representation result by localization. 

\change{\begin{theorem}
Let $X$ be a generalized semimartingale with having characteristics $(A,B,C,\nu)$. Then $X$ can be written as
\[
X=X_0+K+X^c+ (\chi\cdot \text{id})*(\mu^X-\nu)+ \sum_{s\leq \cdot} (\Delta X_s-\chi(\Delta X_s)\cdot\Delta X_s) + B
\]
\end{theorem}
Recall that $\Delta$ can only be reached by $K$ while $\infty$ might be reached by the other parts. The jump measure $\mu^X$ as well as its compensator $\nu$ are defined pre-locally up to $\zeta^\infty$. The same is true for $B$ and $X^c$. }

\begin{example} \label{ex:firstexamples}(1) Let $\widetilde{Z}$ be a L\'evy process with killing. It is a generalized semimartingale by the arguments we have recalled in the Introduction. The first three characteristics are $(\ell t, Qt, dtN(dy))$. Recall that L\'evy processes can not explode. The fourth characteristic of this process is
\[
  A_t= at \hspace{10mm} \text{ on } \hspace{10mm} \csi{0}{\zeta^\Delta}
\]
where $a$ is the fourth component of the L\'evy quadruple as well as the local killing rate (cf. $\lambda^x$ in \eqref{killingrate} below).
This can be seen as follows: the compensator of a Poisson process $(P_t)_{t\geq 0}$ is $(at)_{t\geq 0}$. By stopping at $\zeta^\Delta$ we obtain that 
\[
  ((P_t-at)^{\zeta^\Delta})_{t\geq 0}
\] 
is a local martingale. 

(2) Next we consider spontaneous killing. In case of the process (for $\bbp^x$)
\begin{align*}
  \widetilde{X}_t=x\cdot 1_{\loi{0}{1}}+\Delta 1_{\oi{1}{+\infty}}
\end{align*}
the fourth characteristic is $1_{\{\widetilde{X}_t=\Delta\}}$. We obtain such a compensator whenever the killing process $K$ is predictable.

(3) Let $X$ be the solution of an SDE with locally Lipschitz coefficients. We consider the case where $X$ has explosions. Define $\widetilde{X}:= X+K$ where $K$ is an exponential killing independent of $X$ with killing rate one. This yields a process with a possible transition from $\infty$ to $\Delta$. 

\change{(4) Non-Markovian killing is included in our theory: Let $X$ be any semimartingale. We define the new process 
\[
Y_t:=\begin{cases}  \Delta & \text{if there exists }  t\in [0,1[ \text{ such that }X_t<0 \\ X_t & \text{else.} 
\end{cases}
\]
The killing process is predictable but not Markovian. }
\end{example}  

\change{In order to have an example on how to generalize three-characteristics-results to the new setting, we include the following theorem. The proof is a combination of the classical one (\cite{jacodshir} Theorem II.2.21.) plus ideas from the proof of our Theorem 2.13. 
\begin{theorem}
Let $\widetilde{X}$ be a process with killing and let $\tau_n:=\sigma_n\wedge n$ be the announcing sequence of $\zeta^\infty$. There is equivalence between: \\
(a) $\widetilde{X}$ is a generalized semimartingale with characteristics $(A,B,C,\nu)$.  \\
(b) The following processes are local martingales for each $n$. \\
\hspace*{2mm} (i) $M(h)^{\tau_n}$ where $M(h):=\widetilde{X}(h)-B-\widetilde{X}_0$ \\
\hspace*{2mm} (ii) $(M(h)^jM(h)^k-\widetilde{C}^{jk})^{\tau_n}$ for each $0\leq j,k \leq d$  \\
\hspace*{2mm} (iii) $(g*\mu^{\widetilde{X}^{\tau_n}}-g*\nu)^{\tau_n}$ for $g\in \cc^+(\bbr^d)$  \\
\hspace*{2mm} (iv)  $\left(1_{\{\widetilde{X}_t=\Delta\} }-A_t\right)^{\zeta^\Delta}$\\
where $h=\chi\cdot$id, $\widetilde{C}$ is the modified second characteristic (cf. \cite{jacodshir} Definition II.2.16.). In the same monograph $\cc^+(E)$ resp. $\widetilde{X}(h)$ can be found in II.2.20 resp. II.2.4.
\end{theorem}
Let us emphasize that - along the same lines - various other results on the fourth characteristic can be established which are analogous to those found in \cite{jacodshir} Section II.2. }

\change{\begin{remark}
The theory of stochastic integration can now be generalized as well. It works for the explosion part pre-locally as in the classical case. A-posteriori one includes the instant-killing into the integral process by adding a killing process $K$ having as killing time the minimum of the $\Delta$-killing times of the integrand and the integrator.
\end{remark}
}

\change{From now on we restrict ourselves to the following (still quite general) class of stochastic processes:}

\begin{definition} \label{def:autsemi}
An \emph{autonomous semimartingale} 
$(\widetilde{X},\bbp^x)_{x\in \E}$ is a generalized semimartingale on $\E$ with characteristics $(A,B,C,\nu)$ of the form 
\begin{align} \begin{split} \label{fourchars}
  A_t(\omega)       &=\int_0^t a(\widetilde{X_s}(\omega)) \ ds, \hspace{10mm}\text{ on } \csi{0}{\zeta^\Delta}\\
  B_t^{j}(\omega) &=\int_0^t  \ell^{j}(\widetilde{X_s}(\omega)) \ ds,  \hspace{9mm} \text{ on } \rosi{0}{\zeta^\infty} \text{ for }j=1,...,d \\
  C_t^{jk}(\omega)  &=\int_0^t Q^{jk}(\widetilde{X_s}(\omega)) \ ds,       \hspace{6mm} \text{ on } \rosi{0}{\zeta^\infty} \text{ for }j,k=1,...,d\\
  \nu(\omega;ds,dy) &=N(\widetilde{X_s}(\omega),dy) \ ds \hspace{10.5mm} \text{ on } \rosi{0}{\zeta^\infty}
\end{split} \end{align}
for every $x\in E$ with respect to a fixed cut-off function $\chi$. Here, $a(x)\geq 0$, $\ell(x)=(\ell^{1}(x),...,\ell^{d}(x))'$ is a vector in $\bbr^d$, $Q(x)$ is a positive semi-definite matrix and $N$ is a Borel transition kernel such that $N(x,\{0\})=0$. We call $a$, $\ell$, $Q$ and $n:=\int_{y\neq 0} (1\wedge \norm{y}^2) \ N(\cdot,dy)$ the \emph{differential characteristics} of the process.
\end{definition}

\begin{remark} 
We call the above class `autonomous semimartingale' since each part of the dynamics is driven by the process itself. 
\end{remark}

The following proposition is easily deduced from the definitions above. 

\begin{proposition} \label{prop:seperateone}
Let $\widetilde{X}$ be an autonomous semimartingale. Then $\widetilde{X}$ can be written as $\widetilde{X}=X+K$ where $X$ is a homogeneous diffusion with jumps and explosion, that is, an autonomous semimartingale having $a=0$. The process $K=(K_t)_{t\geq 0}$ attends only the values 0 and $\Delta$ and has the fourth characteristic
\[
  A_t=\int_0^t a(X_s) \ ds.
\] 
on $\csi{0}{\zeta^\Delta}$.
\end{proposition}

Subsequently we analyze the connection between certain Markov processes and autonomous semimartingales. Since the proof of the following result is involved, we have shifted it to Section \ref{sec:proofs}.

We start with Markov processes defined via the martingale problem: let $(a,\ell,Q, n)$ be as in Definition \ref{def:autsemi} and assume from now on all four components to be locally bounded. Then 
\begin{align} \begin{split}\label{generator}
  \ca u(x):= &-a(x)u(x) + \sum_{j=1}^d \ell^{j} (x) \frac{\partial u(x)}{\partial x^{j}} +\frac{1}{2} \sum_{j,k=1}^d Q^{jk}(x) \frac{\partial^2 u(x)}{\partial x^j \partial x^k} \\
	&=\int_{\bbr^d} \Big( u(x+y)-u(x)-\nabla u(x)'y \chi(y) \Big) \ N(x,dy) 
\end{split} \end{align}
defines a linear operator from $C_c^2(E)$ to $B_b(E)$, the space of bounded Borel measurable functions. 

\begin{definition}
We say that a probability measure $\bbp$ on $(\Omega, \cf^X)$ is a solution of the martingale problem for $\ca$, if for all $u\in C_c^2(E)$, 
\[
  M_t^u:=u(X_t)-u(X_0)-\int_0^t \ca u(X_s) \ ds , t\geq 0, 
\]
is a $\bbp$-martingale with respect to $(\cf_t^X)_{t\geq 0}$. We say that the martingale problem for $\ca$ is well-posed, if for every probability distribution $\eta$ on $E$ there exists a unique solution $\bbp^\eta$ of the martingale problem for $\ca$ such that $\bbp^\eta\circ X_0^{-1}=\eta$. 
\end{definition}


Every solution of the martingale problem is an autonomous semimartingale with respect to the filtration made right continuous:

\begin{theorem}\label{thm:markovsemimart}
Let $(\bbp^x)_{x\in \change{E}}$ be a family of solutions to the martingale problem for $\ca$ defined via $(a,\ell, Q, N)$ and $\chi$, such that
\[
  \bbp^x \circ X_0^{-1}=\varepsilon_x,
\]
the Dirac measure in $x$.
After separating $\partial$ into $\Delta$ and $\infty$ as above, $X=(X_t)_{t\geq 0}$ is an autonomous semimartingale on $(\Omega, \cf^X, (\cf^X_{t+})_{t\geq 0}, \bbp^x)_{x\in \change{E}}$ and the characteristics $(A,B,C,\nu)$ with respect to $\chi$ are given by \eqref{fourchars}.
\end{theorem}

The following result shows, that the above theorem encompasses \cite{levymatters3} Theorem 2.44. Compare in this context Hoh \cite{hoh95}. Dealing with Feller processes, there are different conventions in the literature: here, we consider $T_tu(x):=\bbe^x u(X_t) $ $(t\geq 0)$ as semigroup on $(C_0(\bbr^d),\norm{\cdot}_\infty)$, the space of continuous functions vanishing at infinity. 
For every Markov process in the sense of Blumenthal an Getoor, $T_t$ is for every $t\geq 0$ a contractive, positivity preserving
and sub-Markovian operator. If in addition \newline
\hspace*{1cm}(F1) $T_t:C_0(\bbr^d) \pf C_0(\bbr^d)$ for every $t\geq
0$ and \newline \hspace*{1cm}(F2) $\lim_{t\downarrow 0}
\norm{T_tu-u}_\infty =0$ for every $u\in C_0(\bbr^d)$ \newline we
call the semigroup and the associated process $X=(X_t)_{t\geq 0}$
\emph{Feller}.
The generator $(\ca,D(\ca))$ is the closed operator given by
\begin{align} \label{generatordef}
  \ca u(x):=\lim_{t \downarrow 0} \frac{T_t u(x) -u(x)}{t} \hspace{1cm}  (u\in D(\ca))
\end{align}
where $D(\ca)\subseteq C_0(\bbr^d)$ is the set on which the
limit \eqref{generatordef} exists in strong sense, that is,  uniformly in
$x\in\bbr^d$.

\begin{corollary}\label{cor:fellersemimart}
Let $X$ be a Feller process on $\bbr^d$ with  c\`adl\`ag paths. Furthermore, let $C_c^\infty(\bbr^d) \subseteq D(\ca)$. 
After separating $\partial$ into $\Delta$ and $\infty$ as above, $X$ is an autonomous semimartingale on $(\Omega, \cf^X, (\cf^X_{t+})_{t\geq 0}, \bbp^x)_{x\in \bbr^d}$.
\end{corollary}

\begin{proof}
Let $(\ca, D(\ca))$ denote the generator of the Feller process $X$. From the assumption $C_c^\infty(\bbr^d) \subseteq D(\ca)$ it is easily deduced that $C_c^2(\bbr^d)\subseteq D(\ca)$ and that $\ca$ is given by \eqref{generator} on this function space. This result seems to be some kind of folklore anyway, but it can be deduced rigorously as a straight-forward extension of \cite{symbolkillingopenset} Theorem 3.7 and its corollary. By \cite{revuzyor} Proposition VII.1.6 the process $M^u$ is a martingale for every $u\in C_c^2(\bbr^d)$. The local boundedness of the differential characteristic is always fulfilled for rich Feller processes by Lemma 3.3 of \cite{symbolkillingopenset}. In the proof of that result only properties of negative definite functions have been used. These remain true, if a fourth component is considered. Hence, the result follows from our theorem above. 
\end{proof}

We have to impose two assumptions on the differential characteristics $(a,\ell,Q,n)$ in order to derive our second main result. These assumptions are very weak and they are satisfied by virtually every example in the literature. The first assumption is the local boundedness of the differential characteristics, the second assumption reads as follows: 

\begin{definition} \label{def:finelycont}Let $X$ be a generalized semimartingale and $f:\E\to\bbr$ be a Borel-measurable function. $f$ is called $X$\emph{-finely continuous} (or \emph{finely continuous}, for short) if the function
\begin{align} \label{rightcont}
  t\mapsto f(X_t)=f\circ X_t
\end{align}
is right continuous at zero $\bbp$-a.s.
\end{definition}

\begin{remark}
(a) The points $\Delta$ and $\infty$ do not have to be considered in this definition. The process starting in $\Delta$ has to be constant and $\infty$ can be left only by jumping to $\Delta$. 
(b) Fine continuity is introduced differently in the Markovian framework (see the monographs \cite{blumenthalget} Section II.4 and \cite{fuglede}). By Theorem 4.8 of \cite{blumenthalget} the classical definition is equivalent to \eqref{rightcont} in the Markovian setting.
(c) If the differential characteristics are continuous, the condition stated in Definition \ref{def:finelycont} is obviously fulfilled, since the paths of $X$ are c\`adl\`ag.
\end{remark}


Now we are ready to introduce the `symbol' of a stochastic process in this general framework. The symbol offers a neat way to calculate the (extended) generator (if the process is Markovian) and the semimartingale characteristics. Since the symbol contains the same information as the characteristics it has been used (in the conservative case) to analyze e.g. the Hausdorff dimension of paths \cite{schilling98hdd}, their strong variation \cite{withMartynas}, H\"older conditions \cite{schilling98}, ultracontractivity of semigroups \cite{schillingwang}, laws of iterated logarithm \cite{schillingknopova} and stationary distributions of Markov processes \cite{behmeschnurr}. 

Dealing with the symbol, we could work on $E$ with its relative topology. We make things a bit easier by prolonging the process to $\bbr^d$ by setting $X_t:=x$ for $x\in \bbr^d \backslash E$ and $t\geq 0$. Hence, from now on we assume that our processes live on $\bbr^d$ respectively on $\rd=\bbr^d\cup \{\infty,\Delta\}$. Starting with a process on $\E$ local boundedness and fine continuity of the differential charcteristic are not harmed by this extension. 
Furthermore, we write for $\xi\in\bbr^d$
\[
  e_\xi(x):=\begin{cases}e^{i x'\xi}& \text{if } x\in \bbr^d \\ 0 &\text{if } x\in\{ \infty, \Delta \}. \end{cases}
\]

\begin{definition} \label{def:symbol}
    Let $X$ be an $\rd$-valued semimartingale, with respect to $\bbp^x$ for every $x\in\bbr^d$. Fix a starting point $x\in\bbr^d$ and let $K\subseteq \bbr^d$ be a compact neighborhood of $x$. Define $\sigma$ to be the first exit time of $X$ from $K$:
    \begin{gather} \label{stopping}
        \sigma:=\sigma^x_K:=\inf\big\{t\geq 0 : X_t\in \rd \backslash K  \big\}.
    \end{gather}
    The function $p:\bbr^d\times \bbr^d \rightarrow \bbc$ given by
    \begin{gather} \label{stoppedsymbol}
         p(x,\xi):= -\lim_{t\downarrow 0}  \frac{\bbe^x\Big(e_\xi(X_t^\sigma-x)  - 1\Big)}{t}   
    \end{gather}
    is called the \emph{(probabilistic) symbol of the process}, if the limit exists for every $x\in \bbr^d$, $\xi\in\bbr^d$ independently of the choice of $K$.
\end{definition}

If we need the symbol on $\rd$, it is defined as follows: in $\Delta$ it is zero and in $\infty$ it is the local killing rate (starting in $\infty$).

Due to the definition of the symbol the following simple facts hold which we will use several times subsequently
\begin{align} \label{stoppingtimeorder}
  \sigma \leq \zeta^\infty \hspace{15mm} \text{ and } \hspace{15mm} \sigma \leq \zeta^\Delta
\end{align}
where the first inequality is always strict on $\{\zeta^\infty < +\infty\}$. . 
 
Finally we calculate the symbols of the processes under consideration. This concept turns out to be the space-dependent analog of the L\'evy exponent \eqref{lkfa}. \change{Compare in this context \cite{levymatters3} Section 2.5}.

\begin{theorem} \label{thm:stoppedsymbol}
Let $X$ be an autonomous semimartingale on $\widetilde{\bbr^d}$ such that the differential characteristics $a$, $\ell$, $Q$ and $n$ are locally bounded and finely continuous for every $\bbp^x$ $(x\in \bbr^d)$.
In this case the limit \eqref{stoppedsymbol} exists and the symbol of $X$ is 
\begin{align}\label{symbol}
  p(x,\xi)=a(x) -i\ell(x)'\xi + \frac{1}{2} \xi'Q(x) \xi -\int_{y\neq 0} \Big(e^{iy'\xi}-1 -iy'\xi\cdot\chi(y)\Big) \ N(x,dy).
\end{align}
\end{theorem}

\begin{remarks} (a) The symbol allows to calculate the (extended) generator of Markov processes in a neat way using formula \eqref{pseudo}. Even in the non-Markovian case, having calculated the symbol, one can write down the semimartingale characteristics by \eqref{fourchars}. This is nice in particular if the process is given as a solution of an SDE (cf. \cite{sdesymbol}). 

(b) Results on conservativeness of Feller processes or L\'evy-type processes can now be used in full generality (cf. Hoh \cite{walterhabil} Chapter 9, B\"ottcher et al. \cite{levymatters3} Theorem 2.33 and Schilling \cite{schilling98pos} Section 5).

(c) The symbol of an autonomous semimartingale is a state-space dependent continuous negative definite function.
This is natural since the symbol describes the local dynamics of the process. In contrast to L\'evy processes these dynamics depend on the current position in space. 

(d) Let us emphasize that the assumptions of the theorem are very weak. If e.g. the differential characteristics are continuous, they are directly fulfilled.

(e) We believe that various results that have been proved and shown to be useful for classical semimartingales, characteristics and symbols can be transferred to our more general setting.  
\end{remarks}

\section{Complementary Results and Examples}

In order to understand the intuition behind the fourth characteristic, the following example is helpful.

\begin{example}
`just killing'  Consider a time homogeneous Markov process $(X, \bbp^x)_{x\in\bbr}$ with killing which only attends the values $x$ and $\Delta$. Denote the local killing rate by
\begin{align} \label{killingrate}
  \lambda^x:=\lim_{h\downarrow 0} \frac{\bbp^x(X_h = \Delta)}{h}.
\end{align}
It is easy to see that in this case 
\[
  p(x,\xi)=a(x)=\lambda^x
\]
holds.
\end{example}

\begin{example}
`L\'evy process' Let $\widetilde{Z}=Z+K$ be a L\'evy process with killing as considered in the Introduction. We get
\begin{align*}
  -\frac{\bbe^0 e_\xi(X_t^\sigma+K_t) - 1 }{t} 
	   &= -\int_{\{K_t=\Delta\} } \frac{e_\xi(\Delta) -1}{t} d\bbp^0 -\int_{\{K_t\neq \Delta\}} \frac{e^{iX_t'\xi} -1}{t} \ d\bbp^0 \\
     &= \frac{\bbp^0(K_t=\Delta)}{t} + \int 1_{\{K_t\neq\Delta\}} d\bbp^0 \int -\frac{e^{iX_t'\xi} -1}{t} \ d\bbp^0.
\end{align*}
Here, we have used the facts that L\'evy processes are homogeneous in space, stopping does not harm the killing process and $X$ and $K$ are independent in the L\'evy case. 
The last expression tends to the sum of the local killing rate $a$ and the classical L\'evy exponent $\psi$ of $Z$, since $\bbp^0(K_t\neq \Delta)$ tends to one for $t\downarrow 0$. Hence, the symbol of a L\'evy process with killing is its characteristic exponent $\phi$ (cf. \eqref{lkfa}). 
\end{example}
\change{The following results encompases \cite{symbolkillingopenset} Theorem 4.3. It is easily deduced combining Theorem \ref{thm:markovsemimart}, Corollary \ref{cor:fellersemimart} respectively its proof and Theorem \ref{thm:stoppedsymbol}: }

\begin{theorem} 
\change{Let $X$ be a Feller process on $\bbr^d$ with  c\`adl\`ag paths. Furthermore, let $C_c^\infty(\bbr^d) \subseteq D(\ca)$. 
After separating $\partial$ into $\Delta$ and $\infty$ as in Section 2, $X$ is an autonomous semimartingale on $(\Omega, \cf^X, (\cf^X_{t+})_{t\geq 0}, \bbp^x)_{x\in \bbr^d}$. If the differential characteristics are finely continuous, the functional analytic symbol $q(x,\xi)$ and the probabilistic symbol $p(x,\xi)$ coincide for this process.}
\end{theorem}

\begin{example} `Superdrift' \label{ex:superdrift}
Let us consider the deterministic Markov process given by 
\[
X_t=\begin{cases} \frac{1}{\frac{1}{x} -t}& \text{if } t\in
[0,1/x[ \\ \infty &\text{else} \end{cases}
\]
under $\bbp^x$ for $x\geq 0$. For $x<0$ we set $X_t=x$ for $t\geq 0$. This process is an autonomous semimartingale and even a Feller process. 
The symbol of this process is $p(x,\xi)= -ix^2\xi$ and therefore its first characteristic is
$B_t(\omega)=\int_0^t X_s^2 \, ds$. Although the process leaves $\bbr$ the fourth characteristic is zero. 
\end{example}

\begin{example} `CIR with jumps and killing'
The following one-dimensional process $Z$ considered in Cherdito et al. \cite{che-fil-yor} Section 6 is a semimartingale in our extended framework: Let $(W_t)_{t\geq 0}$ be a standard Brownian motion, $(N_t)_{t\geq 0}$ a compound process with jump arrival rate $\lambda >0$ and positive jumps. Let the jumps be given by a probability measure $m$ on $]0,\infty[$. Furthermore, let $\tau$ be an exponentially distributed random time with mean $1/ \gamma >0$. Let $\sigma >0$, $b_0\geq \sigma^2/2$ and $b_1\in \bbr$. Let $Y=(Y_t)_{t\geq 0}$ be given on $]0,\infty[$ as the unique strong solution of 
\begin{align*}
  dY_t&= (b_0+b_1 Y_t) \ dt + \sigma \sqrt{Y_t} \ dW_t + dN_t \\
  Y_0&=y,\quad y>0,
\end{align*}
and $Y_t:=y$ identically on $]-\infty, 0]$. The process $Z$ is then given by
\[
  Z:=Y1_{\rosi{0}{\tau}}+\Delta1_{\rosi{\tau}{+\infty}}
\]
The fourth characteristic of this process is $\lambda t$ on $\csi{0}{\tau}$ and the symbol of $Z$ on $]0,\infty[$ is
\[
  p(x,\xi)=\lambda-i(b_0+b_1x)\xi+\frac{1}{2} \sigma^2 x \xi^2- (\varphi_m(\xi)-1)
\]
where $\varphi_m$ denotes the characteristic function of the measure $m$. The example in  \cite{che-fil}  as well as the affine processes in \cite{che-fil-kim} could be treated in the same way.
\end{example}

\begin{example} `dangerous areas'
Let $Y$ be a given Markov semimartingale without killing. Now we add an additional component: Let $a\in C_b(\bbr^d)$ and define $X$ to be equal to $Y$ with the following exception:
\[
  \bbp^x(X_t\in \Delta)=\bbe^x\int_0^t a(X_s) \change{e^{-a(X_s)s}} \ ds
\]
In this case $a$ is the fourth differential characteristic of $X$. Processes $Y$ used in mathematical finance could be modified in this way. Interpretation: if a company is going through rough times (high value of $a$) for a long time, it is more likely that a sudden bankrupt occurs. The COGARCH process (cf. \cite{cogarch}, \cite{cogarch2}, \cite{cogarchsymbol} ) might be a candidate for such a modification. 

If $Y$ is given as a solution of a L\'evy driven SDE with coefficient $\Phi: \bbr^d \to \bbr^{d \times m}$, the symbol of the modified process $X$ is
\[
p(x,\xi)=a(x)+\psi(\Phi(x)'\xi).
\]
where $\psi$ is the L\'evy exponent of the $m$-dimensional driver (cf. \cite{sdesymbol}, \cite{withMartynas}).
\end{example}

\section{Proofs of the Main Results} \label{sec:proofs}

\begin{proof}[Proof of Theorem \ref{thm:markovsemimart}]
Let $x\in E$ and $\bbp^x$ be a solution to the martingale problem for $\ca$ with $\eta=\varepsilon_x$, the Dirac measure in $x$. 
In the first step we work along the lines of \cite{che-fil-yor} Section 3: we identify $\partial$ with a point $\widehat{\partial}$ in $\bbr^d \backslash \change{E}$. Such a point exists without loss of generality since otherwise we can extend $\ell$, $Q$, $N$ and $\chi$ in a straightforward way to $\bbr^{d+1}$. 

The modified process 
\[
  \widehat{X}:= X1_{\rosi{0}{\zeta^{\partial}}}+\widehat{\partial}1_{\rosi{\zeta^{\partial}}{+\infty}}
\]
is $(\cf_t^X)$-adapted and has right-continuous paths in $\bbr^d$. Nevertheless, $\norm{\widehat{X}_{\zeta^{\partial}-}}=+\infty$ is still possible, that is, the modified process might explode. 

Now let $\sigma$ be an arbitrary $(\cf_t^X)$-stopping time such that $\sigma < \zeta^\partial$. In this case
\[
  \bigcup_{n\geq 1} \{\sigma < \sigma_n\} = \Omega. 
\]
Therefore the local boundedness of $(a,\ell,Q,n)$ implies that the following $(\cf_t^X)$-predictable processes and random measures are well defined for every $\omega\in\Omega$: 
\begin{align*}
  B_t^\sigma&:=\int_0^{\sigma \wedge t} \ell(X_s) \ ds \\
	C_t^\sigma&:=\int_0^{\sigma \wedge t} Q(X_s) \ ds \\
	\nu^\sigma(dt,dy) &:=\Big(N(X_t, dy) + a(X_t) \varepsilon_{\widehat{\partial}-X_t}(dy) \Big) \ dt
\end{align*}
In order to simplify the expression for $B_t^\sigma$ the truncation function is chosen in a way that $\chi(\widehat{\partial}-x)=0$ for all $x\in E$. 

By Proposition 3.2 of \cite{che-fil-yor} for every $(\cf_t^X)$-stopping time $\sigma<\zeta^\infty$ we obtain that $\widehat{X}^\sigma$ is a classical semimartingale on $(\Omega, \cf, \cf_{t+}^X, \bbp^x)$ with characteristics $(B^\sigma, C^\sigma, \nu^\sigma)$ with respect to $\chi$. Here, $\zeta^\infty$ denotes the explosion time, cf. \eqref{killingtimes}. 

Now we proceed as described in Section 2. Let $\sigma$ specifically be the $\sigma_n$ announcing $\zeta^\infty$ ($n\in\bbn$). We set $\infty:=\widehat{\partial}$ for those $\omega$ with $\sigma_n(\omega) < \zeta^\partial(\omega)$ for all $n\in\bbn$. Letting $n\to\infty$ we obtain that $\widehat{X}$ is a semimartingale with explosion having the characteristics $(B^{\sigma_n}, C^{\sigma_n}, \nu^{\sigma_n})$ on $\rosi{0}{\sigma_n}$, in particular these three characteristics are well defined on $\rosi{0}{\zeta^\infty}$. Finally, we include the point $\Delta$ by setting $\Delta:=\widehat{\partial}$ for those $\omega\in\Omega$ such that $\sigma_n(\omega)=\zeta^\partial(\omega)$ for some $n\in\bbn$.  Since it is now possible (again) to jump to $\Delta\notin \bbr^d$ we define in addition
\[
    K_t:=\Delta \cdot 1_{\rosi{\zeta^\Delta}{+\infty}}
\] 
and obtain $X=\widehat{X} + K$. Since $x+\Delta=\Delta$ for every $x\in\bbr^d$ we can delete the last part of the third characteristic. This last part directly yields, however, the structure of the fourth characteristic. Hence, the theorem is proved. 
\end{proof}


Now we prove our second main result. For the reader's convenience we present the one-dimensional proof, the multi-dimensional versions being alike but notationally more involved. 

\begin{proof}[Proof of Theorem \ref{thm:stoppedsymbol}] 
By Proposition \ref{prop:seperateone} the autonomous semimartingale can be written as $\widetilde{X}=X+K$ where $X$ is a semimartingale with explosion and $K$ a killing process with fourth differential characteristic $a$. 

Let us consider:
\begin{align*}
&\bbe^x \left( e^{i(X^\sigma_t+K^\sigma_t-x)\xi} \right)  \\
\hspace*{4mm}&=\int_{\{K^\sigma_t=\Delta\} } e_\xi(X^\sigma_t-x)\cdot e_\xi(K^\sigma_t) \ d\bbp^x
                                  + \int_{\{K^\sigma_t \neq \Delta\} } e_\xi(X^\sigma_t-x)\cdot e_\xi(K^\sigma_t) \ d\bbp^x \\
		\hspace*{4mm}								 & = \int_{\{K^\sigma_t \neq \Delta\} } e_\xi(X^\sigma_t-x) \ d\bbp^x \\
\end{align*}
Therefore, 
\begin{align}\label{splitting}
\frac{\bbe^x e_\xi(\widetilde{X}^\sigma_t-x)-1}{t}
   = \frac{\int_\Omega e_\xi(X^\sigma_t-x) \ d\bbp^x- 1}{t}  -   \frac{\int_{\{K^\sigma_t =\Delta\}} e_\xi(X^\sigma_t-x) \ d\bbp^x}{t}
\end{align}
the first term on the right-hand-side tends to $-p(x,\xi)$, the classical symbol without $a$ by Theorem 3.6  of \cite{generalizedindices}. In this theorem only the conservative case is considered, by \eqref{stoppingtimeorder} this is sufficient. 


For the second term of \eqref{splitting} we obtain:
\[
\frac{\int_{\{K^\sigma_t =\Delta\}} e_\xi(X^\sigma_t-x) \ d\bbp^x}{t} 
   = \frac{1}{t} \int_{\{K^\sigma_t = \Delta \} } 1 \ d\bbp^x  + \frac{1}{t} \int_{\{K^\sigma_t=\Delta \} } \left(e_\xi(X^\sigma_t-x)-1\right) \ d\bbp^x
\]
Here, since the fourth characteristic is defined as a compensator ($K_t^\sigma=\Delta$ iff $\widetilde{X}_t^\sigma=\Delta$), the first term can be written as (cf. \eqref{stoppingtimeorder} and \cite{jacodshir} Theorem I.3.17(ii))
\[
  \frac{1}{t}\bbe^x  A_t^\sigma = \bbe^x \frac{1}{t}\int_0^t a(\widetilde{X}^\sigma_s) \ ds =\bbe^x \int_0^1 a(\widetilde{X}^\sigma_{ts}) \ ds 
\]
which tends to $a(x)$ for t tending to zero, because $a$ is finely continuous. To this end we use the dominated convergence theorem, the fact that $\widetilde{X}^\sigma$ is bounded on $\rosi{0}{\zeta^\infty}$ and that a singleton is a Lebesgue nullset. The subsequent lemma shows that the remainder term tends to zero. Hence, the result. 
\end{proof}

\begin{lemma}
Let $\widetilde{X}$ be an autonomous semimartingale and $\sigma$ as in Definition \ref{def:symbol}.
Separate $\widetilde{X}=X+K$ as in Proposition \ref{prop:seperateone}. 
Furthermore let $u:\rd \to \bbk$ such that $u|_{\bbr^d} \in C_b^2(\bbr^d)$, $u(\infty)=0$ and $u(\Delta)=0$ where $\bbk=\bbr$ or $=\bbc$.
Then 
\begin{align}\label{limitzero}
   \lim_{t \downarrow 0} \frac{1}{t} \bbe^x \Big( 1_{\{K_t^\sigma=\Delta \} }\left( u(X^\sigma_t) -u(x) \right) \Big)= 0
\end{align}
for every $x\in\bbr$
 \end{lemma}
 
This lemma could also be used to derive an alternative proof for Corollary \ref{cor:fellersemimart}. Again we present the one-dimensional proof. 

\begin{proof}   
Let $x\in\bbr$ and $u$ as in the theorem, with values in  $\bbr$.
For the complex valued case separate into real- and imaginary part.
Furthermore, let $M>0$ be such that
\[
  \max\{\norm{u}_\infty, \norm{u'}_\infty, \norm{u''}_\infty \}\leq M< \infty.
\] 
Let the stopping time $\sigma$ be defined as in Definition \ref{def:symbol} where $K$
is an arbitrary compact neighborhood of $x$, such that $K$ is contained in a ball of radius $k$ around zero. 
We use It\^o's formula on the (classical) bivariate semimartingale $(X_t^\sigma, 1_{\{K_t^\sigma =\Delta\} })'$ and the function $(y,z)'\mapsto (u(y)-u(x))\cdot z $. Using $Y_s:=1_{\{K_s^\sigma =\Delta\}} $ as a shorthand we obtain
\allowdisplaybreaks
\begin{align}
\frac{1}{t} &\bbe^x  \Big(\left( u(X^\sigma_t) -u(x) \right) 1_{\{K_t^\sigma=\Delta \} } \Big)\nonumber \\
&= \frac{1}{t} \bbe^x  \left(\int_{0}^t u'(X^\sigma_{s-}) 1_{\{K_s^\sigma=\Delta \} } \ dX^\sigma_s \right) \tag{I}\label{termone}\\
&+ \frac{1}{t} \bbe^x  \left(\int_{0}^t u(X^\sigma_{s-})-u(x) \ dY_s \right)  \tag{II} \label{termtwo} \\
&+ \frac{1}{t} \bbe^x \left(\frac{1}{2}  \int_{0}^t  u''(X^\sigma_{s-})1_{\{K_s^\sigma=\Delta \} } \ d[X^\sigma,X^\sigma]_s^c \right) \tag{III} \label{termthree}\\
&+ \frac{1}{t} \bbe^x \left(\frac{1}{2}  \int_{0}^t  u'(X^\sigma_{s-}) \cdot 1 \ d[X^\sigma,Y]_s^c \right) \tag{IV} \label{termfour}\\
&+ \frac{1}{t} \bbe^x \left(\frac{1}{2}  \int_{0}^t  (u(X^\sigma_{s-})-u(x)) \cdot 0 \ d[Y,Y]_s^c \right) \tag{V} \label{termfive}\\
&+ \frac{1}{t} \bbe^x  \left(\sum_{0\leq s\leq t} \Big(u(X_s^\sigma)-u(X^\sigma_{s-})-u'(X^\sigma_{s-}) \Delta X^\sigma_s \Big)1_{\{K_{s-}^\sigma=\Delta \} }\right) \tag{VI}\label{termsix} \\
&- \frac{1}{t} \bbe^x \left(\sum_{0\leq s\leq t} (u(X^\sigma_{s-})-u(x)) \Delta Y_s\right) \tag{VII} \label{termseven} \\
&+ \frac{1}{t} \bbe^x \left(\sum_{0\leq s\leq t} (u(X_s^\sigma)-u(x)) 1_{\{K_s^\sigma=\Delta\}} - (u(X_s^\sigma)-u(x)) 1_{\{K_{s-}^\sigma=\Delta\}}  \tag{VIII} \label{termeight} \right). 
\end{align}
Term (V) is zero. Term (IV) is zero, too, since $Y$ is a quadratic pure jump semimartingale. Term (II) cancels out with (VII). 
The left-continuous process $X^\sigma_{t-}$ is bounded on $\csi{0}{\sigma}$. Furthermore we have $(\Delta X)^\sigma = (\Delta X^\sigma)$ and $X^\sigma$ admits the stopped characteristics
\begin{equation} \begin{aligned} \label{stoppedchars}
B^\sigma_t(\omega)&=\int_{0}^{t\wedge \sigma(\omega)} \ell(X_s(\omega)) \ ds = \int_0^t \ell (X_s(\omega)) 1_{\csi{0}{\sigma}}(\omega, s) \ ds \\
C_t^\sigma(\omega)&=\int_0^t Q(X_s(\omega)) 1_{\csi{0}{\sigma}}(\omega, s) \ ds \\
\nu^\sigma(\omega;ds,dy)&:=1_{\csi{0}{\sigma}}(\omega, s) \ N(X_s(\omega),dy) \ ds
\end{aligned} \end{equation}
with respect to the fixed cut-off function $\chi$. One can now set the integrand at the right endpoint of the stochastic support to zero, as we are integrating with respect to Lebesgue measure:
\begin{align*}
B^\sigma_t(\omega)&=\int_0^t \ell (X_s(\omega)) 1_{\rosi{0}{\sigma}}(\omega, s) \ ds \\
C_t^\sigma(\omega)&=\int_0^t Q(X_s(\omega)) 1_{\rosi{0}{\sigma}}(\omega, s) \ ds \\
\nu^\sigma(\omega;ds,dy)&=1_{\rosi{0}{\sigma}}(\omega, s) \ N(X_s(\omega),dy) \ ds.
\end{align*}
In the first two lines the integrand is now bounded, because $\ell$ and $Q$ are locally bounded and $\norm{X^{\sigma}_s(\omega)}<k$ on $\roi{0}{\sigma(\omega)}$ for every $\omega \in \Omega$.  
In what follows we will deal with the remaining terms one-by-one.
To calculate the first term we use the canonical decomposition of the semimartingale (see \cite{jacodshir}, Theorem II.2.34) which we write as follows 
\begin{equation} \begin{aligned}
X_t^\sigma=X_0 + X_t^{\sigma,c} &+ \int_0^{t \wedge \sigma} \chi(y)y \ \Big(\mu^{X^\sigma}(\cdot;ds,dy)-\nu^\sigma(\cdot;ds,dy)\Big) \\ &+\check{X}^\sigma(\chi)+B_t^\sigma(\chi).\label{candec}
\end{aligned} \end{equation} 
where $\check{X}_t=\sum_{s\leq t}(\Delta X_s (1- \chi(\Delta X_s))$.
Therefore, term \eqref{termone} can be written as
\begin{align*}
\frac{1}{t} \bbe^x  \Big(\int_{0}^t 1_{\{K_s^\sigma=\Delta \} }u'(X^\sigma_{s-}) \ d\Big( \underbrace{X_t^{\sigma,c}}_{\text{(IX)}} &+ \underbrace{\int_0^{t \wedge \sigma} \chi(y)y \ \Big(\mu^{X^\sigma}(\cdot;ds,dy)-\nu^\sigma(\cdot;ds,dy)\Big)}_{\text{(X)}} \\ &+\underbrace{\check{X}^\sigma(\chi)}_{\text{(XI)}} + \underbrace{B_t^\sigma(\chi) }_{\text{(XII)}}\Big)\Big) 
\end{align*}
We use the linearity of the stochastic integral. Our first step is to prove for term (IX)
\[
  \bbe^x  \int_{0}^t 1_{\{K_s^\sigma=\Delta \} }u'(X_{s-}^\sigma) \ dX_s^{\sigma,c} = 0.
\]
The integral $(1_{\{K_t^\sigma=\Delta \} }u'(X_{t-}^\sigma)) \bullet X_t^{\sigma,c}$ is a local martingale, since $X_t^{\sigma,c}$ is a local martingale. To see that it is indeed a martingale, we calculate the following:
\begin{align*}
&\left[(1_{\{K^\sigma=\Delta \} }u'(X^\sigma)) \bullet X^{\sigma,c}, (1_{\{K^\sigma=\Delta \} }u'(X^\sigma-x)) \bullet X^{\sigma,c} \right]_t \\
&\hspace*{10mm}= \int_0^t (1_{\{K_s^\sigma=\Delta \} }u'(X^\sigma_s))^2 1_{\csi{0}{\sigma}}(s) \ d[X^{c},X^{c}]_s \\
&\hspace*{10mm}= \int_0^t \Big((1_{\{K_s^\sigma=\Delta \} }u'(X^\sigma_s))^2 1_{\rosi{0}{\sigma}}(s)Q(X_s) \Big) \ ds
\end{align*}
where we have used several well known facts about the square bracket. The last term is uniformly bounded in $\omega$ and therefore, finite for every $t\geq 0$.
This means that $(1_{\{K_t^\sigma=\Delta \} } u'(X_t^\sigma)) \bullet X_t^{\sigma,c}$ is an $L^2$-martingale which is zero at zero and therefore, its expected value is constantly zero. \newline
The same is true for the integrand (X). We show that the function $(\omega,s,y)\mapsto 1_{\{K_s^\sigma=\Delta \} }u'(X^\sigma_{s-}) y \chi(y)$ is in the class $F_p^2$ of Ikeda and Watanabe (see \cite{ikedawat}, Section II.3), that is,
\begin{align*}
 \bbe^x \int_0^t \int_{y\neq 0} \abs{u'(X_{s-}^\sigma) \cdot y\chi(y)}^2 \nu^\sigma(\cdot;ds,dy) <\infty.
\end{align*}
To prove this we observe
\begin{align*}
 \bbe^x &\int_0^t \int_{y\neq 0} \abs{1_{\{K_s^\sigma=\Delta \} } u'(X_{s-}^\sigma)}^2 \cdot \abs{y\chi(y)}^2 \nu^\sigma(\cdot;ds,dy) \\
&= \bbe^x \int_0^t \int_{y\neq 0} M^2 \abs{y\chi(y)}^2 1_{\rosi{0}{\sigma}}(\omega, s) N(X_s,dy) \ ds.
\end{align*}
Since we have by hypothesis $\norm{\int_{y\neq 0}(1\wedge y^2) 1_{\rosi{0}{\sigma}}\ N(\cdot,dy)}_\infty < \infty$ this expected value is finite. Therefore, 
\begin{align*}
&\int_0^t  1_{\{K_s^\sigma=\Delta \} }u'(X_{s-}^\sigma) \ d\left(\int_0^{s\wedge \sigma} \int_{y\neq 0} \chi(y)y \ (\mu^{X^\sigma}(\cdot;dr,dy)-\nu^\sigma(\cdot;dr,dy)) \right) \\
&\hspace{10mm}=\int_0^t\int_{y\neq 0} \Big(1_{\{K_s^\sigma=\Delta \} }u'(X_{s-})  \chi(y)y \Big) (\mu^{X^\sigma}(\cdot;ds,dy)-\nu^\sigma(\cdot;ds,dy))
\end{align*}
is a martingale. The last equality follows from \cite{jacodshir}, Theorem I.1.30.

Now we deal with the third term \eqref{termthree}. Here we have
\begin{align*}
[X^\sigma,X^\sigma]_t^c= [X^c, X^c]_t^\sigma=C_t^\sigma &=(Q(X_{t}) \bullet t)^\sigma
= (Q(X_t) \cdot 1_{\rosi{0}{\sigma}}(t))\bullet t
\end{align*}
and therefore,
\begin{align*}
\frac{1}{2} \int_{0}^t (1_{\{K_s^\sigma=\Delta \} }u''(X^\sigma_{s-})) \ d[X^\sigma,X^\sigma]_s^c
&=  \frac{1}{2} \int_{0}^t  1_{\{K_s^\sigma=\Delta \} }u''(X^\sigma_{s-})  Q(X_{s}) \cdot 1_{\rosi{0}{\sigma}}(t) \ ds.
\end{align*}
Since $Q$ is finely continuous and locally bounded we obtain by dominated convergence
\begin{align*}
&\lim_{t\downarrow 0} \frac{1}{2}\frac{1}{t} \bbe^x \int_0^t  1_{\{K_s^\sigma=\Delta \} }u''(X_{s}) Q(X_{s}) 1_{\rosi{0}{\sigma}}(s) \ ds \\
&= \lim_{t\downarrow 0} \frac{1}{2} \bbe^x \int_0^1 1_{\{K_{st}^\sigma=\Delta \} } u''(X_{st}) Q(X_{st}) 1_{\rosi{0}{\sigma}}(st) \ ds =0
\end{align*}
For the finite variation part of the first term, i.e, (XII), we obtain analogously the limit zero.
Finally we plug together the sum in \eqref{termsix} and part (XI) of \eqref{termone}: 
\begin{align*}
& \sum_{0\leq s\leq t}1_{\{K_{s-}^\sigma=\Delta \} } \Big(u(X_s^\sigma)-u(X^\sigma_{s-})-u'(X^\sigma_{s-}) \Delta X^\sigma_s \chi(\Delta X_s^\sigma )\Big) \\
&\hspace{10mm}= \sum_{0\leq s\leq t} 1_{\{K_{s-}^\sigma=\Delta \} }\Big(u(X^\sigma_{s-}+\Delta X_s^\sigma)-u(X^\sigma_{s-})-u'(X^\sigma_{s-}) \Delta X^\sigma_s \chi(\Delta X_s^\sigma )\Big) \\
&\hspace{10mm}=\int_{]0,t]\times \bbr^d} 1_{\{K_{s-}^\sigma=\Delta \} }\Big( u(X^\sigma_{s-}+y)-u(X^\sigma_{s-})-u'(X^\sigma_{s-}) y \chi(y)   \Big) \  \mu^{X^\sigma}(\cdot;ds,dy) \\
&\hspace{10mm}=\int_{]0,t]\times \bbr^d} 1_{\{K_{s-}^\sigma=\Delta \} }\Big( u(X^\sigma_{s-}+y)-u(X^\sigma_{s-})-u'(X^\sigma_{s-}) y \chi(y)   \Big)   \ N(X_s,dy) \ ds.
\end{align*} 
Here, we have used the fact that it is possible to integrate with respect to the compensator of a random measure instead of the measure itself, if the integrand is in $F_p^1$ (see \cite{ikedawat}, Section II.3). This follows by a Taylor expansion and the fact that $u''$ is bounded. 
By the continuity assumption on $N(x,dy)$ we obtain using dominated convergence
\begin{align*}
\lim_{t \downarrow 0} \frac{1}{t} \bbe^x \int_{0}^t \int_{\bbr^d} 1_{\{K_s^\sigma=\Delta\}} \Big( u(X^\sigma_{s-}+y)-u(X^\sigma_{s-})-u'(X^\sigma_{s-}) y \chi(y)  \Big) \ N(X_{s},dy) \ ds  = 0.
\end{align*}
Finally we obtain for \eqref{termeight}
\begin{align*}
&\frac{1}{t} \bbe^x \sum_{0\leq s\leq t} 1_{\{s=\zeta^\Delta\}} (u(X_s^\sigma)-u(x)) \\
&\hspace{10mm}=  \frac{1}{t} \bbe^x  \int_0^t \left(u(X_{s-}^{\zeta^\Delta\wedge\sigma}) -u(x) \right) \ d \left(\Delta(1_{\rosi{\zeta^\Delta}{+\infty}}(s) \right) \\
&\hspace{10mm}=  \frac{1}{t} \bbe^x  \int_0^t \left(u(X_{s-}^{\zeta^\Delta\wedge\sigma}) -u(x) \right) \ d \left( \int_0^s a(X^\sigma_r) \ dr \right) \\
&\hspace{10mm}=  \bbe^x \int_0^1 (u(X_{st}^{\zeta^\Delta\wedge\sigma}) -u(x)) a(X^\sigma_{st}) \ ds
\end{align*}
and hence the limit zero again by dominated convergence. 
\end{proof}

Let us give a short example which shows, that \eqref{limitzero} might fail to be true  outside the class under consideration (even in simple cases).

\begin{example}
Let $E:=\bbr$, $x=0$ and $u\in C_c^\infty(\bbr)$ be equal to the identity on $[-1,1]$. Let $\sigma$ be the first exit time from $[-1,1]$. Define $X_t:=\sqrt{t}$. This deterministic process is independent from  every other stochastic process (and hence from $(K_t)_{t\geq 0}$ which we will define). We obtain for $t\leq 1$:
\[
 \frac{1}{t} \bbe^x 1_{\{K_t^\sigma=\Delta\}} (u(X_t^\sigma)-u(x))= \frac{1}{\sqrt{t}} \cdot \bbp(K_t=\Delta)
\]
Now define $(K_t)_{t\geq 0}$ in a way that for $t_j:=1 / 2^{2j} $ we have $\bbp(K_{t_j}=\Delta)= 1/2^j$. Hence $t_j\downarrow 0$ for $j\to \infty$ and 
\[
   \frac{1}{\sqrt{t_j}} \cdot \bbp(K_{t_j}=\Delta)=1
\]
for every $j\in\bbn$ and the limit \eqref{limitzero} does not exist. 
\end{example}

\textbf{Acknowledgements:} Financial support by the DFG (German Science Foundation) for the project SCHN 1231/2-1 is gratefully acknowledged. The author would like to thank Ph. Protter for the interesting discussions at Columbia University \change{and an anonymous referee whose comments have helped to significantly improve the paper}. 



\end{document}